\documentclass[12pt]{amsart}
\usepackage{epsfig,color}
\usepackage{blindtext}
\usepackage{hyperref}
\usepackage{graphicx}
\usepackage{enumitem} 
\usepackage{url}
\usepackage{amssymb}
\usepackage{graphicx,import}
\usepackage{comment}

\theoremstyle{definition}

\newtheorem{theorem}{Theorem}[section]
\newtheorem{definition}[theorem]{Definition}

\newtheorem{lemma}[theorem]{Lemma}
\newtheorem{remark}[theorem]{Remark}
\newtheorem{corollary}[theorem]{Corollary}

\newtheorem*{remark*}{Remark}

\numberwithin{equation}{section}

\newcommand{\mc}{\mathcal}
\newcommand{\mb}{\mathbb}

\newcommand{\bn}{\mathbf{n}}

\newcommand{\Q}{\mathbb{Q}}

\newcommand{\HH}{\widetilde{H}}
\newcommand{\ZZ}{\widetilde{Z}}
\newcommand{\om}{\omega}
\newcommand{\pa}{\partial}
\newcommand{\mr}{\mathrm}

\newcommand{\R}{\mathbb{R}}

\DeclareMathOperator{\Hess}{Hess}

\title[Bifurcation and genericity]{Bifurcation of perturbations of non-generic closed self-shrinkers}

\date{\today}

\author{Zhengjiang Lin}
\address{Courant Institute of Mathematical Sciences, 251 Mercer Street, New York, NY10012, USA}
\email{malin@nyu.edu}

\author{Ao Sun}
\address{Department of Mathematics,
University of Chicago,
5734 S. University Avenue,
Chicago, IL 60637, USA}
\email{aosun@uchicago.edu}

\begin{document}
\maketitle

\begin{abstract}
    We discover a bifurcation of the perturbations of non-generic closed self-shrinkers. If the generic perturbation is outward, then the next mean curvature flow singularity is cylindrical and collapsing from outside; if the generic perturbation is inward, then the next mean curvature flow singularity is cylindrical and collapsing from inside.
\end{abstract}

\section{Introduction}

A mean curvature flow (MCF) is a family of hypersurfaces $\{M_t\}$ in $\R^{n+1}$ satisfying
\begin{equation}
    \partial_t x=-H \bn.
\end{equation}

Here $H$ is the mean curvature, which is the minus of trace of the second fundamental form, and $\bn$ is the unit outer normal vector. It is known that a mean curvature flow of closed hypersurfaces must generate finite-time singularities, and a central topic in mean curvature flow is to study the singular behavior.

The singularities of mean curvature flows are modeled by a special class of hypersurfaces called self-shrinkers. There are a lot of self-shrinkers (see \cite{A92}, \cite{N14}, \cite{KKM18}) and it seems impossible to classify all self-shrinkers. In contrast, it is believed that, generically, the singularities should not be too complicated (see \cite{H90}, \cite{AIC95}). 

The theory of generic singularities was first developed by Colding-Minicozzi in \cite{CM12}. In \cite{CM12}, Colding-Minicozzi established a way to characterize the genericity of self-shrinkers. Moreover, they showed that only those generalized cylinders $S^m(\sqrt{2m})\times\R^{n-m}$ are generic self-shrinkers. Besides, they proved that one can perturb a non-generic self-shrinker so that it will never be the tangent flow of the mean curvature flow starting at the perturbed hypersurface. 

In this paper, we further investigate the perturbations of non-generic self-shrinkers. Our main theorem is the discovery of the following bifurcation phenomenon:

\begin{theorem}\label{thm:main1}
Suppose $\Sigma^n$ is a non-generic closed embedded self-shrinker in $\R^{n+1}$, with $k$-th homology (in $\mb{Q}$) non-trivial for some $1\leq k\leq n-1$. Then after a generic perturbation (see Definition \ref{D:generic}), the perturbed hypersurface $\widetilde{\Sigma}^n$ under mean curvature flow will
\begin{enumerate}
\item develop a cylindrical singularity collapsing from inside if the perturbation is inward,
\item or develop a cylindrical singularity collapsing from outside if the perturbation is outward.
\end{enumerate}
\end{theorem}

Let us be more precise. A closed hypersurface $\Sigma^n$ in $\R^{n+1}$ would divide $\R^{n+1}$ into two connected components, and we call the unbounded one the {\it outside} of $\Sigma$ and we call the bounded one the {\it inside} of $\Sigma$. Suppose $f$ is a positive function on $\Sigma$, then $s f\bn$ is an {\it outward perturbation} on $\Sigma$ if $s>0$ and an {\it inward perturbation} if $s<0$. A singularity is a {\it cylindrical} singularity if the tangent flow at this singularity is a multiplicity $1$ generalized cylinder $\{\sqrt{-t}(S^m\times\R^{n-m})\}_{t\in(-\infty,0)}$. By Brakke/White's regularity theory (see \cite{W05}), it means that the parabolic blow up sequence at this singularity converges to $\{\sqrt{-t}(S^m\times\R^{n-m})\}_{t\in(-\infty,0)}$ smoothly. Then we say this cylindrical singularity {\it collapsing from inside} if the unit normal vectors of the blow-up sequence converge to the unit outer normal vector on $\{\sqrt{-t}(S^m\times\R^{n-m})\}_{t\in(-\infty,0)}$, and we say this cylindrical singularity {\it collapsing from outside} if the unit normal vectors of the blow-up sequence converge to the opposite of the unit outer normal vector on $\{\sqrt{-t}(S^m\times\R^{n-m})\}_{t\in(-\infty,0)}$. The following figure gives an intuitive illustration of our main theorem.

\begin{figure}[h!]
\includegraphics[width=0.8\linewidth]{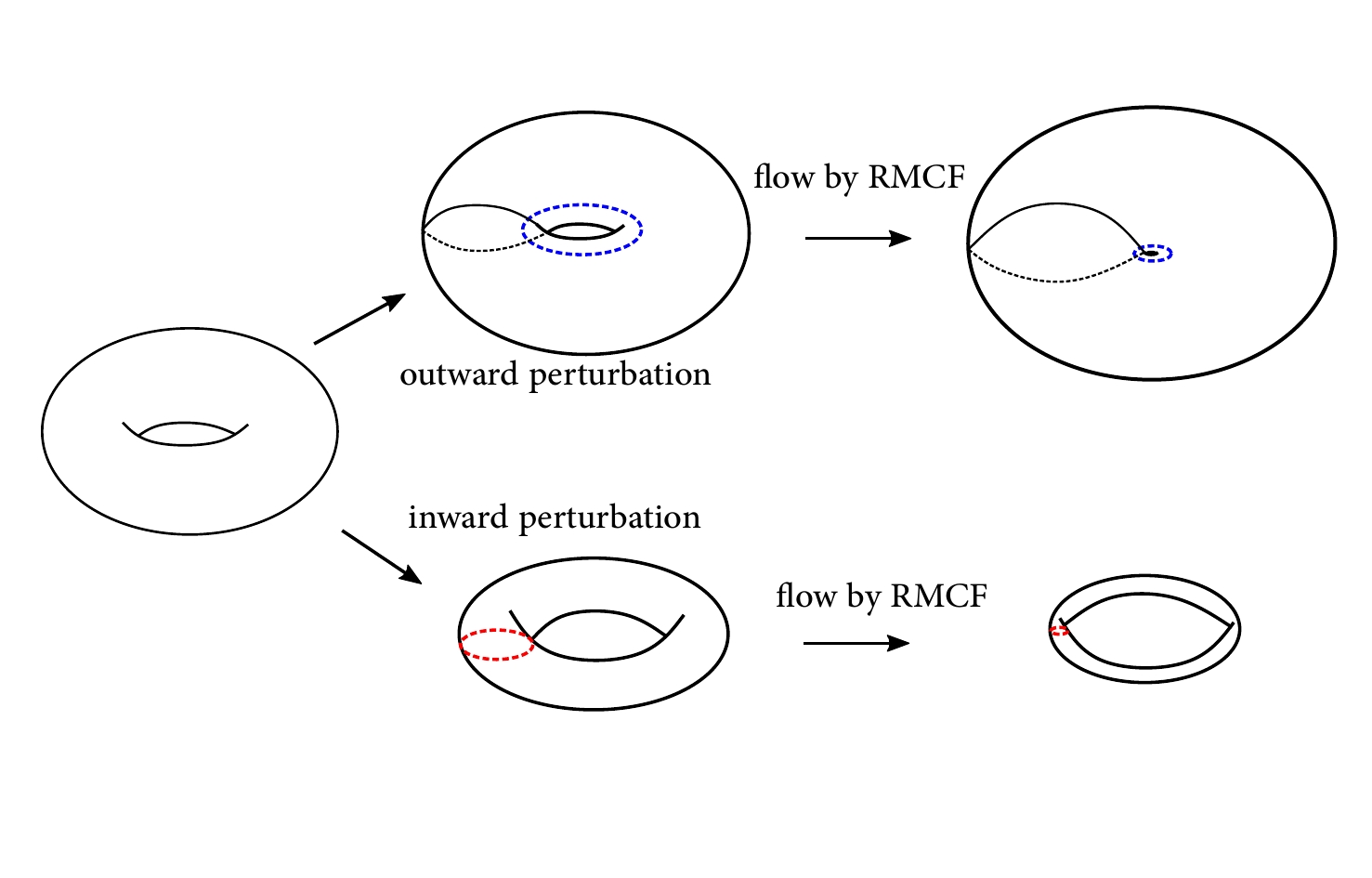}
\caption{Bifurcation of generic perturbations. After an outward perturbation, under RMCF, the blue dotted circle collapses, which represents a collapsing from outside. After an inward perturbation, the red dotted circle collapses, which represents a collapsing from inside.}
\end{figure}

\bigskip

The proof of Theorem \ref{thm:main1} relies on several observations. The first one is a regularity theory developed by Hershkovitz-White \cite{HW19} for rescaled mean curvature flow (In \cite{HW19} they called it $X$-weak flow). A family of hypersurfaces is a rescaled mean curvature flow if they satisfy the equation
\begin{equation}
\partial_t x=-\left(H-\frac{\langle x,\bn\rangle}{2}\right)\bn.
\end{equation}
The rescaled mean curvature flow (RMCF) was introduced by Huisken in \cite{H90}, and each RMCF is equivalent to a mean curvature flow up to a rescaling in spacetime. The quantity $(H-\frac{\langle x,\bn\rangle}{2})$ is called the {\it rescaled mean curvature}. We say a hypersurface is {\it rescaled mean convex } or {\it rescaled mean concave} if $(H-\frac{\langle x,\bn\rangle}{2})>0$ or $(H-\frac{\langle x,\bn\rangle}{2})<0$ respectively (see \cite[Section 2]{CIMW13}).

Suppose $\Sigma$ is a non-generic closed self-shrinker. After a generic outward perturbation, the perturbed hypersurface $\widetilde\Sigma$ is rescaled mean convex or rescaled mean concave (see \cite[Lemma 1.2]{CIMW13}). The parabolic maximum principle shows that a RMCF starting at $\widetilde\Sigma$ must be rescaled mean convex or rescaled mean concave respectively in the future. Moreover, this RMCF is {\it nested}, in the sense that for $0\leq t<s$, $\widetilde M_s$ always lies inside/outside of $\widetilde M_t$ if the RMCF is rescaled mean convex/rescaled mean concave respectively.

Hershkovitz-White have developed the regularity of the limit flow of a rescaled mean convex RMCF in \cite{HW19} (See also references [7], [8], [9] in \cite{HW19}). Their result also holds true for rescaled mean concave RMCF by flipping the inside and outside of the hypersurfaces. We state their regularity theorem here.

\begin{theorem}[Theorem 4 in \cite{HW19}]\label{T:Smooth convergence}
The singularities of rescaled mean convex/rescaled mean concave are of convex type. That is, the limit flow is smooth and has multiplicity $1$. Moreover, the tangent flows are generalized cylinders.
\end{theorem}

An alternative approach to understand the regularity of the limit flow is to compare it to the regularity theory of mean convex MCF. After the pioneer work by White on mean convex MCF (see \cite{W94}, \cite{W00}, \cite{W03}), there are some interpretations of the results from different points of view. One of them is the non-collapsing analysis initiated by Sheng-Wang in \cite{SW09} and Andrews in \cite{A12}. Haslhofer-Kleiner \cite{HK17} used the non-collapsing condition to give some simple arguments towards the regularity of mean convex MCF (see also \cite{HK17-2}). We prove that a similar non-collapsing argument also holds true for RMCF.
\begin{theorem}\label{thm:main2}
    Let $M^n$ be a closed manifold, and $F: M^n \times [0,T) \to \mb{R}^{n+1}$ a family of smooth embeddings evolving by the rescaled mean
    curvature flow, with positive rescaled mean curvature (negative rescaled mean curvature resp.). If $\widetilde M_0 = F(M, 0)$ is rescaled-$\delta$-non-collapsed (rescaled-$-\delta$-non-collapsed resp.) for some $\delta > 0 $, then $\widetilde M_t = F(M, t)$ is 
    also rescaled-$\delta$-non-collapsed (rescaled-$-\delta$-non-collapsed resp.) for every $t \in [0,T) $. 
\end{theorem}

We would also like to address two alternative approaches related to this problem. The first one was studied by Huisken-Sinestrari in \cite{HS99}, \cite{HS99-2} (also see some applications by Brendle-Huisken in \cite{BH16}, \cite{BH18}). They used differential geometry techniques to study mean convex MCF. Though we do not use this approach in our paper, it is plausible that their techniques can also be used to study RMCF.

The second one was studied by Colding-Ilmanen-Minicozzi-White in \cite{CIMW13}. They studied the regularity of the tangent flows of RMCFs with low entropy. Their approach was further generalized by Bernstein-Wang \cite{BW16} and Zhu \cite{Z16}.

\medskip

Our further observation is a finite time singularity argument for RMCF of the generic perturbed hypersurfaces. Note that a RMCF defined for time $t\in[0,\infty)$ is naturally corresponding to a MCF defined for time $t\in[-1,0)$ (see Section \ref{S:Preliminaries}). Suppose that $\Sigma$ is a non-generic closed self-shrinker, with non-trivial $k$-th homology (in $\mb{Q}$). After a generic perturbation, the perturbed hypersurface $\widetilde\Sigma$ must develop a finite time singularity under RMCF. This is equivalent to say that the corresponding MCF develops a singularity at time $t<0$. Since a self-shrinking MCF develops a singularity at time $t=0$, this means that after the generic perturbation, the MCF starting at $\widetilde{\Sigma}$ would develop a singularity earlier than the time when the self-shrinking MCF generates a singularity.

If the perturbation is inward, the finite-time singularity of RMCF has been proved by Colding-Ilmanen-Minicozzi-White in \cite{CIMW13}, and there is no topological assumption on the self-shrinkers. Nevertheless, for an outward perturbation, the topological assumption that the $k$-th homology is non-trivial is necessary. To see this, we can think about a self-shrinking sphere $S^n(\sqrt{-2nt})$. If we perturb the $-1$ time slide $S^n(\sqrt{2n})$ outwards a little bit to $S^{n}(\sqrt{2n+\epsilon})$, it would not shrink to a point before time $0$. Hence, the RMCF starting at $S^{n}(\sqrt{2n+\epsilon})$ does not have a finite time singularity.

The condition that the $k$-th homology is non-trivial is also used by White-Heshkovitz in \cite{HW19}. It is natural to imagine that this condition is very special for self-shrinkers. We remark that Brendle proved in \cite{B16} that in $\mathbb{R}^3$, non-generic closed embedded self-shrinkers must have a non-trivial $1$-st homology group. 

\medskip

Our final observation is the bifurcation in perturbations. This bifurcation has been discovered for $1$-dimensional MCF. In \cite{AL86}, Abresch-Langer conjectured that given a closed immersed self-shrinker in the plane (which is known as an {\it Abresch-Langer curve}), after an outward perturbation, it would become round under MCF, and after an inward perturbation, it would generate some cusp singularities. This conjecture was proved by Au in \cite{A10}. 

The local dynamic of an Abresch-Langer curve has been studied by Epstein-Weinstein in \cite{EW87}, and Au's result can be viewed as a partial extension to long-time dynamic. Similarly, a higher dimensional local dynamic result has been studied by Colding-Minicozzi in \cite{CM18}, \cite{CM18-2}. Our result can also be viewed as a partial extension of the long-time dynamic. That is, we can give information about the next time singularity after a generic perturbation.

\medskip

We organize this paper as follows. In Section \ref{S:Preliminaries}, we present some preliminaries on rescaled mean curvature flow, self-shrinkers, and generic perturbations. In Section \ref{S:Finite time blow up of RMCF}, we prove a finite time blow-up theorem of RMCF. In Section \ref{S:Non-collapsing}, we study the non-collapsing result of RMCF. In Section \ref{S:Bifurcation}, we study the bifurcation of the next time singularity after the perturbation. We also have an appendix including some computations and the non-collapsing result for RMCF in case the readers may find it interesting.

\bigskip

{\bf Acknowledgement}
A.S. is grateful to his advisor Bill Minicozzi for support and encouragement. A.S. wants to thank Zhichao Wang, Jinxin Xue, Jonathan Zhu for discussions and comments. Z.L. would like to thank Professor Fang-Hua Lin for encouragement. Both authors are grateful to the anonymous referee for helpful comments.

\section{Preliminaries}\label{S:Preliminaries}
A self-shrinker in $\R^{n+1}$ is a hypersurface $\Sigma$ satisfying the equation
\begin{equation}
H-\frac{\langle x,\bn\rangle}{2}=0.
\end{equation}
The name comes from the fact that $\{\sqrt{-t}\Sigma\}_{t\in(-\infty,0)}$ is a mean curvature flow. Recall that the tangent flow of a MCF at spacetime point $(y,T)$ is the limit of the parabolic rescaling sequence $\{\alpha_i(M_{T+\alpha_i^{-2}t}-y)\}$ as $\alpha_i\to\infty$. It was proved by Huisken \cite{H90} (cf. \cite{W94}, \cite{I95}) that a tangent flow is a self-shrinking MCF $\{\sqrt{-t}\Sigma\}_{t\in(-\infty,0)}$ where $\Sigma$ is a self-shrinker.

In this paper, we only study closed embedded self-shrinkers. Then, by Alexander duality, the self-shrinker would divide $\R^{n+1}$ into two connected components. Recall from the introduction that we call the bounded one {\it inside} and the unbounded one {\it outside}. We will always fix the unit normal vector $\bn$ at each point on $\Sigma$ to be the one pointing outside. 

A self-shrinker is {\it generic} if and only if it is $S^k(\sqrt{2k})\times \R^{n-k}$ for $k\in\{0,1,\cdots,n\}$. This definition is related to some deep theory in \cite{CM12}. 

We refer the readers to \cite[Section 2]{CM12} for further discussions on self-shrinkers. Here we discuss an important differential operator on self-shrinkers and generic self-shrinkers. Given a self-shrinker $\Sigma$, the {\it linearized operator} $L$ is defined by
\begin{equation}
    Lu=\Delta_\Sigma u -\frac{1}{2}\langle x,\nabla^\Sigma u\rangle +\left(1/2+|A|^2\right)u.
\end{equation}
The self-shrinkers are critical points of the Gaussian area functional
\[F(\Sigma)=\int_\Sigma e^{-|x|^2/4}d\mu_\Sigma,\]
and $L$ is the second variational operator for this functional. It is a self-adjoint operator with respect to the Gaussian density $e^{-|x|^2/4}d\mu_\Sigma$.

\medskip

Inspired by the shrinker's equation, we define the {\it rescaled mean curvature} of a hypersurface in $\R^{n+1}$ to be the quantity
\begin{equation}
    \widetilde H= H-\frac{\langle x,\bn\rangle}{2}.
\end{equation}
Recall from the introduction that we say a hypersurface is {\it rescaled mean convex} if $\widetilde H>0$ and we say a hypersurface is {\it rescaled mean concave} if $\widetilde H<0$.

\medskip

Next, we study the perturbations of a self-shrinker. We will use the following notation: if $\Sigma$ is a closed hypersurface in $\R^{n+1}$ and $f$ is a function on $\Sigma$, then we define the graph of $f$ over $\Sigma$ to be the hypersurface
\[\Sigma^f:=\{x+f(x)\bn(x):x\in\Sigma\}.\]

In Lemma 1.2 of \cite{CIMW13}, Colding-Ilmanen-Minicozzi-White perturbed a closed self-shrinker by the first eigenfunction of the linearized operator $L$. Then they proved that after the perturbation, the perturbed hypersurface is rescaled mean convex or rescaled mean concave, depending on the direction of the perturbation. We state the lemma here.

\begin{lemma}[Lemma 1.2 of \cite{CIMW13}]
  Let $\Sigma$ be a non-generic self-shrinker. There exists a positive function $f$ and $\epsilon>0$ such that $\Sigma^{sf}$ is 
 \begin{itemize}
 \item rescaled mean convex if $-\epsilon<s<0$,
 \item rescaled mean concave if $0<s<\epsilon$.
 \end{itemize}
\end{lemma}

In \cite{CM12}, Colding-Minicozzi proved the existence of such a perturbation. This perturbation could lower a quantity called {\it entropy}, which is a core argument in \cite{CM12}. In this paper, we do not need any arguments directly related to entropy (though actually there are some relations, see \cite{CM12}), and we define the generic perturbations as follows.

\begin{definition}\label{D:generic}
 We say a perturbation $f\bn$ is a {\it generic perturbation} if 
 \begin{itemize}
 \item either $\Sigma^{f}$ is rescaled mean convex, and in this case we say $f\bn$ is an {\it inward perturbation};
 \item or $\Sigma^{f}$ is rescaled mean concave, and in this case we say $f\bn$ is an {\it outward perturbation}.
 \end{itemize}
\end{definition}

\medskip
Next, we study the {\it rescaled mean curvature flow} (RMCF). Recall from the introduction that a rescaled mean curvature flow is a family of hypersurfaces satisfying the equation
\begin{equation}\label{D:Curve under RMCF}
 \partial_t x= -\left(H-\frac{\langle x,\bn\rangle}{2}\right)\bn.
\end{equation}

A RMCF is equivalent to a MCF up to a rescaling in spacetime. More precisely, suppose $\widetilde M_t$ is a RMCF defined for $t\in[0,\infty)$, then $M_\tau=\sqrt{-\tau}\widetilde M_{-\log(-\tau)}$ is a MCF defined for $\tau\in[-1,0)$. Conversely, if $M_\tau$ is a MCF defined for $\tau\in[-1,0)$, then $\widetilde M_t=M_{-e^{-t}}/\sqrt{e^{-t}}$ is a RMCF defined for $t\in[0,\infty)$.

For MCF, Huisken \cite{H84} has done some important computations, which show that many geometric quantities satisfy some parabolic evolution equations. Later in \cite[Lemma 3.1]{CIMW13}, Colding-Ilmanen-Minicozzi-White did similar computations for RMCF. Here we list some evolution equations of geometric quantities. They play important roles in our later study of the non-collapsing results. See Lemma \ref{L:Evolution of the intrinsic geometry}, Lemma \ref{L:Evolution of the extrinsic geometry} and Lemma \ref{L:Evolution of T}. Here we only point out a simple but important fact based on these computations.

\begin{lemma}\label{L:Preservation of rescaled mean curvature}
    For a RMCF we have the following equation
    \begin{equation}
        \partial_t \left(H-\frac{\langle x,\bn\rangle}{2}\right)=L \left(H-\frac{\langle x,\bn\rangle}{2}\right).
    \end{equation}
    Moreover, by parabolic maximum principle, a RMCF is rescaled mean convex/rescaled mean concave if the initial hypersurface is rescaled mean convex/rescaled mean concave respectively.
\end{lemma}

\section{Finite time blow up of RMCF}\label{S:Finite time blow up of RMCF}
The goal of this section is to prove the following finite time blow-up of the perturbed hypersurface under RMCF. The proof is similar to the proof of finite-time singularity in \cite{HW19}, and we generalize its argument to outward perturbations.

\begin{theorem}\label{T:finite time singularity}
    Suppose $\Sigma$ is an $n$-dimensional closed smoothly embedded self-shrinker in $\R^{n+1}$. Moreover, $\Sigma$ is not generic and has a non-trivial $k$-th homology (in $\mb{Q}$) class for some $1\leq k\leq n-1$. Let $f$ be a positive function on $\Sigma$. Then, there is an $\epsilon_0>0$ such that for $|s|<\epsilon_0$, $s\neq 0$, the RMCF starting at $\Sigma^{sf}$ has a finite time singularity.
\end{theorem}

According to the classification theorem of \cite{B16}, the only genus $0$ closed embedded self-shrinker in $\R^3$ is the sphere with radius $2$, which is generic. Thus, in $\R^3$, the topological assumption holds true automatically.

\begin{corollary}
        Suppose $\Sigma$ is a $2$-dimensional non-generic closed smoothly embedded self-shrinker in $\R^{3}$. Let $f$ be a positive function on $\Sigma$. Then there is an $\epsilon_0>0$ depending on $f$ such that for $|s|<\epsilon_0$, $s\neq 0$, the RMCF starting at $\Sigma^{sf}$ has a finite time singularity.
\end{corollary}

We need some lemmas to prove Theorem \ref{T:finite time singularity}. First, we notice that Theorem \ref{T:finite time singularity} is equivalent to a short time blow-up property of MCF. Recall that if $\widetilde M_t$ is a RMCF defined for $t\in[0,\infty)$, then $M_\tau = \sqrt{-\tau}\widetilde M_{-\log(-\tau)}$ is a MCF, defined for $\tau\in[-1,0)$ (see Section \ref{S:Preliminaries}).

\begin{lemma}\label{L:RMCF finite time equivalent to MCF short time}
Let $\widetilde M_t$ be a RMCF. Then $\widetilde M_t$ having a finite time singularity is equivalent to MCF $M_\tau=\sqrt{-\tau}\widetilde M_{-\log(-\tau)}$ having a singularity at time $\tau=T$ with $T<0$.
\end{lemma}

From now on we will fix a positive perturbation function $f$. The RMCF starting at $\Sigma^{sf}$ is denoted by $\widetilde M^s_t$, and the corresponding MCF $\sqrt{-\tau}\widetilde M^s_{-\log(-\tau)}$ is denoted by $M^s_\tau$. $M^s_\tau$ is a MCF starting at time $-1$. We only need to prove that $M^s_\tau$ has a singularity before time $0$.

\medskip
Next, we recall the avoidance principle of MCF of closed hypersurfaces.

\begin{lemma}\label{L:avoidance}
Suppose $\{M_\tau ^1\}_{\tau \in[0,T)}$ and $\{M_\tau^2\}_{\tau\in[0,T)}$ are two MCFs of smoothly embedded closed hypersurfaces. If $M_0^1\cap M_0^2=\emptyset$, and the distance between $M_0^1$ and $M_0^2$ is $d$, then the distance between $M_\tau^1$ and $M_\tau^2$ is at least $d$.
\end{lemma}

The proof is a standard application of the parabolic maximum principle. We refer the readers to \cite[Theorem 2.1.1]{M11} for proof.

\begin{remark}
The avoidance principle is also used by White \cite{W00} to define a weak form of mean curvature flow. See also \cite{HK17}.
\end{remark}

\medskip
We need some topological arguments. Let us recall some facts in algebraic topology which might be well-known to the experts. For the fundamental algebraic topology results, we refer the readers to \cite{H02}.

Let $K$ be a $k$-dimensional closed submanifold in $\R^{n+1}$ and $L$ be a compact, locally contractible topological subspace in $\R^{n+1}$ with $K \cap L = \emptyset$. Since they are both compact, we can assume that they both lie in a big ball $B_R$ of $\R^{n+1}$. Then we can compactify $\R^{n+1}$ to get a sphere $S^{n+1}$. Then we say that $K$ has a non-trivial linking with $L$ if $[K]\in H_{k}(S^{n+1}\backslash L)$ is non-trivial. The non-trivial linking property is an isotopic invariance.

Now we can prove the main theorem of this section. The idea is similar to the proof of Theorem 2 in \cite{HW19}, with which one could show the existence of finite-time singularity under both inward and outward perturbations. The proof for outward perturbations needs more topological ingredients.

\begin{proof}[Proof of Theorem \ref{T:finite time singularity}]
We prove the case of $s>0$. The case $s<0$ has been proved in \cite{CIMW13}, and can be proved similarly using our argument here. Define $A$ to be the closure of the inside of $\Sigma^{sf}$, $A'$ to be the closure of the inside of $\Sigma$, and define $B$ to be the closure of the outside of $\Sigma$, $B'$ to be the closure of the outside of $\Sigma^{sf}$. Then $A\cap B=\bigcup_{0\leq a\leq s}\Sigma^{af}$, which is topologically $\Sigma\times[0,s]$. The assumption that $\Sigma$ has a nontrivial $k$-th homology implies that we can pick a $k$-th cycle $K\subset A\cap B$ which is non-trivial in $H_k(A\cap B)$. Since we are working in $\Q$ coefficient homology, we can always pick this $K$ to be an embedded $k$-dimensional submanifold (see Section 4 of \cite{S04}). 

$A\cup B$ is the whole $\R^{n+1}$, and after a compactification at infinity, it is $S^{n+1}$. By Mayer–Vietoris sequence, we have the following long exact sequence 
\[\cdots
\rightarrow
H_{k+1}(S^{n+1})
\rightarrow
H_k(A\cap B)
\rightarrow
H_k(A)\oplus H_k(B)
\rightarrow
H_k(S^{n+1})
\rightarrow
\cdots
\]
Then for $1\leq k\leq n-1$, we obtain an isomorphism $H_k(A\cap B)\cong H_k(A)\oplus H_k(B)$. Then $[K]$ is also non-trivial in $H_k(A)$ or $H_k(B)$.

If $[K]$ is non-trivial in $H_k(A)$, since $K\subset A\cap B$, we can pick $K\subset\Sigma$. Then $[K]$ is also non-trivial in $H_k(S^{n+1}\backslash B^i)$, where $B^i$ is the interior of $B$. Moreover, there is a deformation retraction from $B^i$ to $B'$. Thus, $[K]$ is also non-trivial in $H_k(S^{n+1}\backslash B')$. In other words, $K$ has a non-trivial linking with $B'$. Besides, if the distance between $\Sigma^{sf}$ and $\Sigma$ is at least $d$, then the $d/2$-tubular neighbourhood $N_{d/2}K$ does not intersect with $B'$.

Suppose $M_\tau^s$ is the MCF starting at $\Sigma^{sf}$. We use $K_\tau$ to denote $\sqrt{-\tau}K$, and $B'_\tau$ to denote the outside of $M_\tau^s$. Now we argue by contradiction. If $M_\tau^s$ does not have a singularity before $\tau=0$, then $\sqrt{-\tau}\Sigma$ and $M_\tau^s$ are all isotopies (parametraized by $t$) of hypersurfaces by the avoidance principle, Lemma \ref{L:avoidance}. Moreover, Lemma \ref{L:avoidance} implies that $B'_\tau$ does not intersect with $N_{d/2}K_\tau$. However, when $\tau$ is sufficiently close to $0$, $N_{d/2}K_\tau$ is contractible (it is actually star-shaped), and then $K_\tau$ cannot represent a non-trivial homology  class in $H_k(S^{n+1} \backslash B'_\tau) \cong H_k(S^{n+1}\backslash B')\cong H_k(S^{n+1}\backslash B^i)$, which is a contradiction.

If $[K]$ is non-trivial in $H_k(B)$, a similar argument works. Now we can pick $K$ in $\Sigma^{sf}$, and $N_{d/2}K$ has a non-trivial linking with $A'$. Again we can argue by contradiction to show that $M_\tau^s$ must have a singularity before time $\tau=0$. This concludes the proof.
\end{proof}

\section{Non-collapsing}\label{S:Non-collapsing}
In this section, we study the non-collapsing property of a RMCF starting at a perturbed hypersurface. The main regularity theorem is due to \cite{HW19} and we refer the readers to there. We show that the RMCF and its blow-up sequence satisfy a non-collapsing property similar to the non-collapsing property of a mean convex MCF, and discuss some consequences. The precise computations are very similar to \cite{A12}, and we include them in the appendix for the convenience of the readers.

Let us first review some basic facts about the tangent flow of MCF with additional forces.

\begin{theorem}\label{T:self-shrinking tangent flow}
    Suppose $M_t$ is a MCF with $L^\infty$ additional forces. Then the tangent flow of $M_t$ is a weak homothetic shrinking MCF, i.e. it is given by $\{\sqrt{-t}\Sigma\}_{t\in(-\infty,0)}$, where $\Sigma$ is an integral varifold satisfying the self-shrinker's equation
    \[\vec{H}+\frac{x^\bot}{2}=0.\]
\end{theorem}

If the additional force is zero, i.e. $M_t$ is a MCF, this theorem was first proved by Huisken in \cite{H90} with an extra curvature assumption. Later White \cite{W94} and Ilmanen \cite{I95} proved this theorem for a MCF without curvature assumptions. The proof for a MCF with $L^\infty$ additional force is similar, and we refer the readers to \cite{S18} for proof. We remark that a RMCF of closed hypersurfaces can be viewed as a MCF with $L^\infty$ additional forces.

\bigskip

In the rest of this section, we study the non-collapsing result of a RMCF.

\begin{definition}
    Given $\delta>0$. A rescaled mean convex hypersurface $M$ bounding an open region $\Omega$ in $\mb{R}^{n+1}$ is {\it rescaled-$\delta$-non-collapsed} if
    for every $x \in M$ there is an open ball $B$ of radius $\delta/\HH$ contained in $\Omega$ with $x \in \partial \Omega$.
\end{definition}

This definition is a natural generalization of the definition of non-collapsing of a mean convex hypersurface in \cite{SW09} and \cite{A12}. Note that given a closed hypersurface $M\subset\R^{n+1}$, a furthest point on $M$ must have positive mean curvature. Thus there is no closed mean concave hypersurface in $\R^{n+1}$. Nevertheless, there are many closed rescaled mean concave hypersurfaces, and they play important roles in this paper. So we also define the following types of non-collapsing for these hypersurfaces.

\begin{definition}
    Given $\delta<0$. A rescaled mean concave hypersurface $M$ bounding an open region $\Omega$ in $\mb{R}^{n+1}$ is {\it rescaled-$ \delta$-non-collapsed} if
    for every $x \in M$ there is an open ball $B$ of radius $\delta / \HH$ contained in $\Omega$ with $x \in \partial \Omega$.
\end{definition}

Given a hypersurface $M$, we define a function on $M \times M$ by
    \begin{align*}
        \ZZ (x,y) = \frac{\HH}{2} \| y - x \|^2 + \delta \big\langle y-x , \bn(x) \big\rangle . 
    \end{align*}
Then we have the characterization:
\begin{lemma}\label{L:Equivalence of nonnegative Z tilde}
    Given $\delta>0$. A rescaled mean convex hypersurface $M$ is rescaled-$\delta$-non-collapsed if and only if $\ZZ \geq 0$ for all $x,y \in M$.
    A rescaled mean concave hypersurface $M$ is rescaled-$-\delta$-non-collapsed if and only if $\ZZ \leq 0$ for all $x,y \in M$.
\end{lemma}
\begin{proof}
    The proof is the same as in \cite{A12}.
\end{proof}

Similar to \cite{A12}, we can prove the following non-collapsing theorem.

\begin{theorem}\label{T:Shrinker delta non collapsed}
    Let $M^n$ be a closed manifold, and $F: M^n \times [0,T) \to \mb{R}^{n+1}$ be a family of smooth embeddings evolving by the rescaled mean
    curvature flow, with positive rescaled mean curvature (negative rescaled mean curvature resp.). If $\widetilde M_0 = F(M, 0)$ is rescaled-$\delta$-non-collapsed (rescaled-$-\delta$-non-collapsed resp.) for some $\delta > 0 $, then $\widetilde M_t = F(M, t)$ is 
    also rescaled-$\delta$-non-collapsed (rescaled-$-\delta$-non-collapsed resp.) for every $t \in [0,T) $. 
\end{theorem}

The proof is based on Andrews' proof in \cite{A12} with necessary modification. For completeness, we provide proof in the Appendix. 

\begin{remark}\label{R:RMCF blow up}
    If $F:M^n \times [0,T) \to \mb{R}^{n+1}$ is a RMCF defined as in (\ref{D:Curve under RMCF}), then, for any $\sigma \in \mb{R}_{+}$ and $y \in \mb{R}^{n+1}$, $T\in(0,\infty)$, the parabolic rescaling of the RMCF   $F^{\sigma,(y,T)}(x,t)=\sigma (F(x,T+\sigma^{-2} t)-y)$ satisfies the evolution equation:

    \begin{equation}\label{E:Evolution of RMCF after rescaling}
        \pa_t F^{\sigma,(y,T)} = -\left(H - \frac{1}{2\sigma^2}\langle F^{\sigma,(y,T)},\bn \rangle-\frac{\langle y,\bn\rangle}{2\sigma}\right)\bn \ .
    \end{equation}

This evolution equation also corresponds to the first variation of weighted area with density $e^f$ for $f = - \frac{|x-\sigma y|^2}{4\sigma^2}$. A version of Theorem \ref{T:Shrinker delta non collapsed} for a parabolic rescaling of RMCF also holds true, see Remark \ref{R: Rescaling non collapsed}.

\end{remark}

\begin{remark}
    Remark 5 in \cite{A12} suggests that a RMCF preserves not only the non-collapsing property from ``inside" region bounded by the closed hypersurface, but also the non-collapsing from ``outside" region which is not bounded by the closed hypersurface. So the computation actually yields a two-sided non-collapsing along the RMCF.
    
    From now on, we will use {\it rescaled-$|\delta|$-non-collapsing} to denote the non-collapsing from the inside and the outside at the same time.
\end{remark}

Rescaled-$|\delta|$-non-collapsing (from both inside and outside) implies the 
following curvature pinching result.

\begin{corollary}
    Suppose $M_0$ is a closed embedded rescaled mean convex/rescaled mean concave hypersurface in $\R^{n+1}$, which is rescaled-$|\delta|$-non-collapsed from both inside and outside. Then we have the following curvature pinching estimate
    \begin{equation}
        -|\delta|^{-1}|\HH| g\leq A\leq |\delta|^{-1}|\HH| g,
    \end{equation}
    where $A$ is the second fundamental form and $g$ is the metric tensor. 
\end{corollary}

\medskip

With non-collapsing of RMCF and its parabolic dilation sequence (see Remark \ref{R:RMCF blow up} and Remark \ref{R: Rescaling non collapsed}), we could follow Haslhofer-Kleiner's arguments in \cite{HK17} to study Theorem \ref{T:Smooth convergence}. The idea is very similar to the idea in Haslhofer-Kleiner's \cite{HK17}. Though the proof in \cite{HK17} is quite elegant and short, it is still too long to fit here. Instead, we will sketch the basic ideas here and refer the readers to \cite{HK17} for detailed proofs.

The key is a half-space convergence theorem (Theorem 2.1 in \cite{HK17}). A necessary modification is to set a sequence of parabolic rescalings of RMCF, which are rescaled-$|\delta|$-non-collapsed flow, with their $\sigma$ in (\ref{E:Evolution of RMCF after rescaling}) tending to $\infty$, and $y$ in (\ref{E:Evolution of RMCF after rescaling}) uniformly bounded. Hence, the evolution equation (\ref{E:Evolution of RMCF after rescaling}) is like the one of MCF as $\sigma\to\infty$ and $y$ uniformly bounded. One also needs to get the evolution of spheres under RMCF, which could be derived from the evolution of spheres under MCF and the $1-1$ correspondence between RMCF and MCF. In order to show further that this convergence is smooth, one also has to establish the local density estimates by one-sided minimization property with the gaussian area.

By this half-space convergence of RMCF type, we can argue as in \cite{HK17} that one could also get the corresponding curvature estimate. After which the arguments are the same as in Haslhofer-Kleiner's \cite{HK17}.

\medskip
So far, we have not specified whether the tangent flow in Theorem \ref{thm:main1} is the tangent flow of the MCF starting at $\widetilde\Sigma$, or the tangent flow of the RMCF starting at $\widetilde\Sigma$. At the end of this section, we prove that they are actually the same objects.

\begin{lemma}\label{Lemma;same tangent flows}
Let $\widetilde{M}_t$ be a RMCF, which has a singularity at the spacetime point $(y,T)$. Let $M_\tau$ be the corresponding MCF, which has a singularity at $(e^{-\frac{T}{2}}y,-e^{-T})$. Then one tangent flow of the RMCF at $(y,T)$ is the same as one tangent flow of the MCF at $(e^{-\frac{T}{2}}y,-e^{-T})$.
\end{lemma}

\begin{remark}
Note that the tangent flow at a singularity may not be unique. So in the statement of the lemma, we only prove that one tangent flow of the RMCF is the same as one tangent flow of the MCF. Nevertheless, in our later application, because the tangent flow is a self-shrinking cylinder, the tangent flow is unique by the work of Colding-Minicozzi \cite{CM15}. 
\end{remark}

\begin{proof}[Proof of Lemma \ref{Lemma;same tangent flows}]
Recall that if $\widetilde M_t$ is a RMCF defined for $t\in[0,\infty)$, then $M_\tau=\sqrt{-\tau}\widetilde M_{-\log(-\tau)}$ is a MCF defined for $\tau\in[-1,0)$. Then the time $-1$ slice of the blow up sequence of the RMCF is
\[\alpha_i (\widetilde M_{T-\alpha_i^{-2}}-y)=\alpha_i\left(\frac{M_{-e^{\alpha_i^{-2}-T}}}{e^{\alpha_i^{-2}/2-T/2}}-y\right)=\beta_i(M_{-e^{-T}+\beta_i^{-2}t_i}-e^{-\frac{T}{2}}y)+\alpha_i(e^{-\alpha_i^{-2} /2}-1)y.\]
Here $\beta_i=\alpha_i e^{-\alpha_i^{-2}/2+T/2}$ and $t_i=\alpha_i^{2}(e^{-\alpha_i^{-2}}-1)$. Hence it is also a blow up sequence of the corresponding MCF with a slightly variance in time and a translation in the space. As $\alpha_i\to\infty$,  we have $\beta_i\to \infty$, $t_i\to -1$ and $\alpha_i(e^{-\alpha_i^{-2} /2}-1)\to 0$. Moreover, the limit of the blow up sequence of RMCF and MCF are both homothetically shrinking MCFs. Therefore, the tangent flow of $\widetilde M_t$ at $(y,T)$ must be identified with the tangent flow of $M_\tau$ at $(e^{-\frac{T}{2}}y,-e^{-T})$.
\end{proof}

\section{Bifurcation}\label{S:Bifurcation}
In this section, we study the bifurcation phenomenon. Recall from Remark \ref{R:RMCF blow up} that the tangent flow of a RMCF $\widetilde M_t$ at spacetime point $(y,T)$ is the limit of the sequence
\[
\widetilde{M}^{\sigma,(y,T)}_t=\sigma (\widetilde{M}_{T+\sigma^{-2}t}-y),
\]
and $\widetilde{M}^{\sigma,(y,T)} _t$ satisfies the equation
\[
\pa_t x^{\sigma,(y,T)} = -\left(H - \frac{1}{2\sigma^2}\langle x^{\sigma,(y,T)},\bn \rangle-\frac{\langle y,\bn\rangle}{2\sigma}\right)\bn,
\]
where $x^{\sigma,(y,T)}$ is the position of $\widetilde{M}^{\sigma,(y,T)}_t$.

Now suppose $\widetilde M_t$ is a RMCF starting at a perturbed hypersurface given by an inward perturbation of a non-generic closed self-shrinker, with $k$-th non-trivial homology. Theorem \ref{T:finite time singularity} shows that $\widetilde M_t$ has a finite time singularity, and we assume it is $(y,T)$. Lemma \ref{L:Preservation of rescaled mean curvature} implies that $\widetilde M_t$ is rescaled mean convex for every time $t<T$, hence $\langle \pa_t x,\bn\rangle<0$. Since $\widetilde{M}^{\sigma,(y,T)}_t$ is a parabolic rescaling of $\widetilde M_t$, it also satisfies $\langle \pa_t x^{\sigma,(y,T)},\bn\rangle<0$. Therefore we obtain
\[
H\geq  \frac{1}{2\sigma^2}\langle x^{\sigma,(y,T)},\bn \rangle+\frac{\langle y,\bn\rangle}{2\sigma}.
\]

By passing to a limit, the regularity theory of Hershkovitz-White \cite{HW19} (see also the discussion in Section \ref{S:Non-collapsing}) shows that $\widetilde M^{\sigma,(y,T)}_t$ smoothly converges to a multiplicity $1$ self-shrinking mean curvature flow. Let $\sigma\to \infty$, we notice that $H\geq 0$ on the limit. Therefore the classification theorem by Huisken \cite{H90} (see also \cite{W00}, \cite{W03} and \cite{CM12}) implies that the tangent flow must be a self-shrinking (generalized) cylinder.

A similar argument holds for RMCF starting at a perturbed hypersurface given by an outward perturbation of a non-generic closed self-shrinker, with $k$-th non-trivial homology. In this case, a similar argument shows that after passing to a limit, the tangent flow satisfies 
$H\leq 0$. Therefore, by flipping the sign we obtain that the tangent flow must be a self-shrinking (generalized) cylinder.

In particular, the sign of $H$ of the blow-up sequence implies two different collapsing directions. This verifies the bifurcation phenomenon.

\appendix
\section{Computation of RMCF}\label{A:Computation of RMCF}

\subsection{Evolution Equations}

\indent We need to clarify some notations we would like to use in our calculations. Given a local coordinate, the second fundamental forms are
$h_{ij} = \langle\nabla_{\pa_i} \bn , \pa_j \rangle \ $, the mean curvature is
$H = h_{ij}g^{ji}\ $. We denote by
$T = \langle\nabla f , \bn\rangle\ $, and 
$ \HH = H + T\ $.
Here $f$ is a smooth function on $\mb{R}^{n+1}$. Our rescaled MCF (RMCF) is defined by:
    $$F: M^n \times [0,T) \to \mb{R}^{n+1}\ ,$$
with
    \begin{equation}\label{D:Definition of RMCF}
        \frac{d F}{d t}(x,t) = - \HH_x \bn_x \ .
    \end{equation}
The subscription shows at which point the geometric quantities are defined. Also, we have used the abbreviation that $\pa_i = \frac{dF}{dx_i}$ and $\pa_t = \frac{dF}{dt}$. $\nabla$ is the usual
Levi-Civita connection of $\mb{R}^{n+1}$ while $\nabla^M$ is the induced connection of $M$ at time $t$. $\bn$ is the unit outer normal vector.

\begin{lemma}[Evolution of the intrinsic geometry]\label{L:Evolution of the intrinsic geometry}
    \begin{equation}
        \pa_t g_{ij} = -2\HH h_{ij} \ ,
    \end{equation}
    \begin{equation}
        \pa_t \mr{det} g = - 2 \HH  H \  \mr{det}g \ . 
    \end{equation}
\end{lemma}

\begin{lemma}[Evolution of the extrinsic geometry]\label{L:Evolution of the extrinsic geometry}
    \begin{equation}
        \pa_t \bn = \nabla^M \HH\ ,
    \end{equation}
    \begin{equation}
        \pa_t h_{ij} = (\Hess_M  \HH) _{ij} - \HH (h^2)_{ij} \ ,
    \end{equation}
where $(h^2)_{ij} = h_{ip} g^{pq} h_{qj}$, and $\Hess_M$ is the Hessian on $M$.
    \begin{equation}
        \pa_t H = \Delta_M \HH + |h|^2 \HH \ .
    \end{equation}
    
\end{lemma}

\begin{lemma}[Evolution of T]\label{L:Evolution of T}
    \begin{equation}
        \pa_t T = - \HH \nabla^2 f (\bn , \bn) + \langle\nabla f , \nabla^M \HH\rangle \ .
    \end{equation}
    
\end{lemma}

All these three lemmas are fundamental calculations. One could turn to \cite{H84} for the evolution equations of MCF, and the calculations are very similar.

\subsection{{Main Calculation}}
\indent Most of our calculation would follow \cite{A12} with the evolution equations of RMCF above. 

In order to prove Theorem \ref{T:Shrinker delta non collapsed}, by Lemma \ref{L:Equivalence of nonnegative Z tilde},  we consider the following funtion:
    \begin{align*}
        \ZZ(x,y,t) = \frac{\HH(x,t)}{2} \| F(y,t) - F(x,t) \|^2 + \delta \langle F(y,t) - F(x,t) , \bn(x,t) \rangle.
    \end{align*}
Then, the theorem is equivalent to prove that this function is non-negative everywhere provided that it is non-negative on $M \times M \times \{0\}$.
For our convenience, we would use same notations as in \cite{A12}. We would still apply maximum principle to show this. Let $\HH _x$ be the rescaled mean curvature and $\bn_x$ the outward unit normal vector at $(x,t)$, and let
$d = |F(y,t) - F(x,t)| $ and $\om = \frac{F(y,t) - F(x,t)}{d} $, and $\pa_i ^x = \frac{\pa F}{\pa x^i}$.
\\
\indent We compute first and second derivatives of $\ZZ$, with respect to some choices of local normal coordinates $\{x^i\}$ near $x$ and $\{y^i\}$ near $y$. In our following calculation, $T$ is in a general form, not just in the situation that $f = - \frac{|x|^2}{4}$. We only
use $f = - \frac{|x|^2}{4}$ in the final proof of Theorem \ref{T:Shrinker delta non collapsed}. Moreover, the same result of Theorem \ref{T:Shrinker delta non collapsed} also holds if we rescale $f$, see Remark \ref{R: Rescaling non collapsed} .
\begin{lemma}[First derivatives of $\ZZ$]\label{L:First derivatives of Z tilde}
    \begin{equation}\label{E:Derivative of Z tilde wrpt y}
        \frac{\pa \ZZ}{\pa y^i} = d \HH_x \langle \om , \pa_i ^y \rangle + \delta \langle \pa_i ^y , \bn_x \rangle,
    \end{equation}
    \begin{equation}
        \frac{\pa \ZZ}{\pa x^i} = -d \HH _x \langle \om , \pa_i ^x \rangle + \frac{d^2}{2}\nabla_i \HH_x + \delta d h_{iq} ^x g_{x} ^{qp} \langle \om, \pa_p ^x \rangle ,
    \end{equation}
    \begin{equation}
        \begin{split}
            \frac{\pa \ZZ }{ \pa t} = & d \HH_x \langle \om, - \HH _y \bn_y + \HH_x \bn_x \rangle + \frac{d^2}{2} \big[\Delta_M \HH_x + \HH_x |h^x|^2  + \pa_t(T_x) \big] 
                              \\    &+ \delta \langle - \HH_y \bn_y + \HH_x \bn_x , \bn_x \rangle + \delta d \langle \om , \nabla^M \HH_x \rangle ,
        \end{split}
    \end{equation}
    where $h_{ij} ^x = \langle \nabla_{\pa_i ^x} \bn_x , \pa_j ^x \rangle$ is the second fundamental form at $(x,t)$.
\end{lemma}

\begin{proof}
    The first two equations are by direct calculation, while in the third one, we use Lemma \ref{L:Evolution of the extrinsic geometry} to express $\pa_t H$ and $\pa_t \bn_x$
\end{proof}
We then have the following relation among $\bn_x$, $\bn_y$ and $\om$, and the proof is the same to the one in \cite{A12}.
\begin{lemma}\label{L:General relation among normal vectors}
    \begin{equation}
        \bn_x + \frac{d \HH_x}{\delta} \om - \frac{1}{\delta} \frac{\pa \ZZ}{\pa y^q}g_y ^{qp} \pa_p ^y = \pm \bn_y \sqrt{1 + \frac{2 \HH_x}{\delta^2} \ZZ - \frac{1}{\delta^2} | \nabla^M _y \ZZ |^2} .
    \end{equation}
\end{lemma}

\begin{proof}
    We first show that the left hand side is normal to each $\pa_i ^y$:
        $$\langle \pa_i ^y, \bn_x + \frac{d \HH_x}{\delta} \om - \frac{1}{\delta} \frac{\pa \ZZ}{\pa y^q}g_y ^{qp} \pa_p ^y \rangle = \langle \pa_i ^y , \bn_x + \frac{d \HH_x }{\delta} \om \rangle - \frac{1}{\delta}\frac{\pa \ZZ}{\pa y^i} = 0 \ .$$
    The last equality is from (\ref{E:Derivative of Z tilde wrpt y}) of Lemma \ref{L:First derivatives of Z tilde}. Hence, the left hand side is parallel to $\bn_y$. Then,
        \begin{equation*}
                        \| \bn_x + \frac{d \HH_x}{\delta} \om - \frac{1}{\delta} \frac{\pa \ZZ}{\pa y^q}g_y ^{qp} \pa_p ^y \|^2 
                = 1 + \frac{2 \HH_x}{\delta^2} \ZZ - \frac{1}{\delta^2} | \nabla^M _y \ZZ |^2 \ ,
        \end{equation*}
    by the same calculation in \cite{A12}.
\end{proof}

\begin{remark}
    Note that we use $\pm$ here as we only calculated the length of the left-hand side of Lemma \ref{L:General relation among normal vectors}. But no matter $+$ or $-$ here, at the final step, each '$-$' term would multiply with another '$-$' term. Hence, they would cancel each other. It loses nothing if one takes $+$ or $-$ in Lemma \ref{L:General relation among normal vectors}.
\end{remark}

\begin{lemma}[Second derivatives of $\ZZ$ at $(x,y)$]
    \begin{equation}
    \left\{
    \begin{aligned}
        &\frac{\pa^2 \ZZ}{\pa y^i \pa y^j} = \HH_x \langle \pa_i ^y, \pa_j ^y \rangle - d\HH_x h_{ij} ^y \langle \om , \bn_y \rangle - \delta h_{ij} ^y \langle \bn_y, \bn_x \rangle , 
        \\
        &\frac{\pa^2 \ZZ}{\pa y^i \pa x^j} = - \HH_x \langle \pa_i ^y , \pa_j ^x \rangle + d \langle\om , \pa_i ^y \rangle \nabla_j \HH_x + \delta h_{jq} ^x g_{x} ^{qp} \langle \pa_i ^y, \pa_p ^x \rangle , 
        \\
            &\frac{\pa^2 \ZZ}{\pa x^j \pa x^i} = \HH_x \langle \pa_j ^x , \pa_i ^x \rangle -d\langle \om , \pa_i ^x \rangle \nabla_j \HH_x + d \HH_x h_{ij} ^x  \langle \om , \bn_x \rangle - d \langle \om , \pa_j ^x \rangle \nabla_i \HH_x 
            \\    & \qquad\qquad
            + \frac{d^2}{2} \nabla_j \nabla_i \HH_x
            + \delta d (\nabla_j h_{iq} ^x) g_x ^{qp} \langle \om , \pa_p ^x \rangle - \delta h_{ij} ^x - \delta d h_{iq} ^x g_x ^{qp} h_{pj} ^x \langle \om , \bn_x \rangle \end{aligned}\right.
    \end{equation}
\end{lemma}

\begin{proof}
    The proof is also by directly differentiating equations in Lemma \ref{L:First derivatives of Z tilde}. As we chose local normal coordinates near $x$ and $y$, terms like $\nabla _{\pa_i ^x}^M \pa_j ^x$ and $\nabla _{\pa_i ^y}^M \pa_j ^y$ all vanish at the point $(x,y)$ in the calculation
\end{proof}
Also, we can choose local normal coordinates at first so that $\{\pa_i ^x\}$ is orthonormal at $x$ and $\{\pa_i ^y \}$ is orthonormal at $y$, and $\pa_i ^x = \pa_i ^y$ for $i = 1, \cdots , n-1$. Thus $\pa_n ^x$ and $\pa_n ^y$ are coplanar with $\bn_x$ and $\bn_y$. 
We also require that the orientations formed by $\{\pa_n ^x , \bn_x \}$ and $\{\pa_n ^y , \bn_y \}$ are the same.
\\
\indent For abbreviation, let 
    $$\mc{L}(\ZZ) = \sum\limits _{i,j = 1} ^n g_y ^{ij} \frac{\pa^2 \ZZ}{\pa y^i \pa y^j} + g_x ^{ij} \frac{\pa^2 \ZZ}{\pa x^i \pa x^j} + 2 g_x ^{ik} g_y ^{jl} \langle \pa_k ^x , \pa_l ^y \rangle   \frac{\pa^2 \ZZ}{\pa x^i \pa y^j} \ .$$
Now, by the Codazzi equation $\nabla_j h_{ik} ^x = \nabla_i h_{jk} ^x$, we have:
    \begin{align*}
        \pa_t\ZZ - \mc{L}(\ZZ) =& \   d \HH_x \langle \om, - \HH _y \bn_y + \HH_x \bn_x \rangle + \frac{d^2}{2} \big[\Delta_M \HH_x + \HH_x |h^x|^2  + \pa_t(T_x) \big] 
        \\        &+ \delta \langle - \HH_y \bn_y + \HH_x \bn_x , \bn_x \rangle  +\delta d \langle \om , \nabla^M \HH _x \rangle - n \HH _x + d \HH _x H_y \langle \om , \bn_y \rangle + \delta H_y \langle \bn_y , \bn_x \rangle
        \\        &- n\HH_x - d \HH_x H_x \langle \om , \bn_x \rangle + 2d\langle \om , \nabla^M \HH_x \rangle - \frac{d^2}{2} \Delta_M \HH_x -\delta d \langle \om , \nabla^M H_x \rangle + \delta H_x 
        \\        &+ \delta d  \langle \om , \bn_x \rangle |h^x|^2 + 2(n-1)\HH_x + 2 \HH_x \langle \pa_n ^y , \pa_n ^x \rangle^2 - 2 d \langle \pa_i ^x , \pa_j ^y \rangle \langle \om , \pa_j ^y \rangle \nabla_i \HH_x 
        \\        &- 2\delta ( H_x - h_{nn} ^x + \langle\pa_n ^x ,\pa_n ^y \rangle^2 h_{nn}^x) \ .
    \end{align*}
    We can further simplify the right hand side to
    \begin{align*} 
    & =\ d \HH_x \langle \om , - T_y \bn_y\rangle + d \HH_x \langle \om , T_x \bn_x \rangle + \frac{d^2}{2}\pa_t(T_x) + |h^x|^2 \ZZ +\delta \langle -T_y \bn_y , \bn_x \rangle + \delta T_x
        \\        &+ \delta d \langle \om , \nabla^M T_x \rangle - 2 \HH_x  + 2d \big\langle \om , \pa_i ^x - \langle \pa_i ^x ,\pa_j ^y \rangle \pa_j ^y \big\rangle \nabla_i \HH_x + 2 \HH_x \langle \pa_n ^y ,\pa_n ^x\rangle ^2
        \\        &+ 2\delta h_{nn} ^x (1 - \langle \pa_n ^x , \pa_n ^y \rangle ^2) \ ,
    \end{align*}
    and finally we have
    \begin{align*}
        &= \ \ZZ |h^x|^2 + 2d \big\langle \om , \pa_i ^x - \langle \pa_i ^x ,\pa_j ^y \rangle \pa_j ^y \big\rangle \nabla_i \HH_x - 2 (\HH_x - \delta h_{nn} ^x) (1 - \langle \pa_n ^y ,\pa_n ^x\rangle ^2)
        \\        &+ d \HH_x \langle \om , T_x \bn_x - T_y \bn_y \rangle + \delta \langle T_x \bn_x - T_y \bn_y , \bn_x \rangle + \frac{d^2}{2} \pa_t(T_x) + \delta d \langle \om , \nabla^M T_x \rangle  \ .
    \end{align*}
We denote the first line of the right hand side of the last equation by $A_1$, while the second line of the right hand side of the last equation by $A_2$, i.e.
    \begin{equation}
        A_1 = \ZZ |h^x|^2 + 2d \big\langle \om , \pa_i ^x - \langle \pa_i ^x ,\pa_j ^y \rangle \pa_j ^y \big\rangle \nabla_i \HH_x - 2 (\HH_x - \delta h_{nn} ^x) (1 - \langle \pa_n ^y ,\pa_n ^x\rangle ^2) \ ,
    \end{equation}
    \begin{equation}\label{E:Expression of A_2}
        A_2 = d \HH_x \langle \om , T_x \bn_x - T_y \bn_y \rangle + \delta \langle T_x \bn_x - T_y \bn_y , \bn_x \rangle + \frac{d^2}{2} \pa_t(T_x) + \delta d \langle \om , \nabla^M T_x \rangle \ .
    \end{equation}
Note that $A_1$ is similar to the one Andrews got in \cite{A12}, and we would use the same way to simplify $A_1$ at any critical point of $\ZZ$ (i.e. $\frac{\pa \ZZ}{\pa x^i} = 0$ and $\frac{\pa \ZZ}{\pa y^i} = 0 $ for every $i$).

\begin{lemma}\label{L:Good part}
    \begin{equation}
        A_1 = \bigg(|h^x| ^2 + \frac{4\HH (\HH - \delta h_{nn} ^x)}{\delta^2} \langle\om , \pa_n ^y \rangle ^2\bigg)\ZZ
    \end{equation}
    at any critical point of $\ZZ$.
\end{lemma}

\begin{proof}
    At critical points of $\ZZ$, by Lemma \ref{L:First derivatives of Z tilde}, we have two equations:
    \begin{equation}
        d \HH_x \langle \om , \pa_i ^y \rangle = - \delta \langle \pa_i ^y , \bn_x \rangle \ ,
    \end{equation}
    \begin{equation}\label{E:Gradient of H tilde}
        \nabla_i \HH_x = \frac{2}{d} \langle \om , \HH_x \pa_i ^x - \delta h_{ip} ^x \pa_p ^x \rangle \ .
    \end{equation}
    From Lemma \ref{L:General relation among normal vectors}, at critical points, we also get:
    \begin{equation}\label{E:Relation among normal vectors}
        \bn_x + \frac{d \HH_x}{\delta}\om = \pm \bn_y \rho \ ,
    \end{equation}
    where $\rho = \sqrt{1 + \frac{2\HH_x}{\delta^2} \ZZ }$. This tells us that $\om$ is also in the plane formed by $\{\pa_n ^x , \bn_x\}$.
    Also, as $\pa_n ^x, \pa_n ^y , \bn_x ,\bn_y$ are in the same plane, we have $\pa_n ^x - \langle \pa_n ^x , \pa_n ^y \rangle \pa_n ^y = \langle \pa_n ^x , \bn_y \rangle \bn_y$.
    Then, $1 - \langle \pa_n ^x ,\pa_n ^y \rangle ^2 = \langle \pa_n ^x , \bn_y \rangle ^2$. Hence,
    \begin{align*}
        A_1 =& \  \ZZ |h^x|^2 + 4 \big\langle\om, \pa_n ^x - \langle\pa_n ^x , \pa_n ^y \rangle\pa_n^y \big\rangle \big\langle \om , \HH_x \pa_n^x - \delta h_{nn}^x \pa_n ^x \big\rangle
            - 2(\HH_x - \delta h_{nn}^x)\langle \pa_n ^x ,\bn_y \rangle ^2
        \\  =& \  \ZZ |h^x|^2 + 2(\HH_x - \delta h_{nn}^x) \bigg[2 \langle \om,\bn_y \rangle \langle \pa_n ^x , \bn_y \rangle \langle \om , \pa_n ^x \rangle - \langle \pa_n ^x , \bn_y \rangle ^2 \bigg] \ .
    \end{align*}
    Now, by \ref{E:Relation among normal vectors}, and the fact that $\om$ has unit norm, we get:
    \begin{equation}\label{E:inner product omega nu_Y}
        \begin{split}
                \langle \om , \bn_y \rangle =& \  \pm \frac{1}{\rho} \bigg(  \langle \om , \bn_x \rangle + \frac{d \HH_x}{\delta} \bigg) = \pm  \frac{1}{\rho}\bigg( \frac{1}{ \delta d} \big( \ZZ - \frac{d^2}{2} \HH_x \big) + \frac{d \HH_x}{\delta} \bigg) 
            \\    =& \  \pm \frac{1}{\rho} \frac{d \HH_x}{2 \delta} \bigg(1 + \frac{2 \ZZ}{d^2 \HH_x}\bigg) \ ,
        \end{split}
    \end{equation}
    and
    \begin{equation}
        \langle \om , \pa_n ^x \rangle = \pm \rho \frac{\delta}{d \HH_x}\langle \bn_y , \pa_n ^x \rangle \ .
    \end{equation}
    Hence,
    \begin{align*}
        A_1 &= \ZZ | h^x | ^2+ 2 (\HH_x - \delta h_{nn} ^x) \bigg[\big(1 + \frac{2 \ZZ}{d^2 \HH_x} \big) \langle \bn_y, \pa_n ^x \rangle ^2 - \langle \bn_y, \pa_n ^x \rangle ^2\bigg]
        \\  &= \bigg(|h^x| ^2 + \frac{4 (\HH_x - \delta h_{nn} ^x)}{d^2 \HH_x} \langle \bn_y , \pa_n ^x \rangle ^2 \bigg) \ZZ = \bigg(| h^x | ^2 + \frac{4 \HH_x (\HH_x - \delta h_{nn} ^x)}{\delta ^2} \langle \om , \pa_n ^y \rangle ^2\bigg) \ZZ \ .
    \end{align*}
\end{proof}

Till now, what we have obtained is very similar to the one we can get from Andrews's work \cite{A12}, but we have a bad term $A_2$.
We would simplify $A_2$ in a general form $T_x = \langle \nabla f_x ,\bn_x \rangle$ first and then specify $f = - \frac{|x|^2}{4}$ to prove Theorem \ref{T:Shrinker delta non collapsed}.

\begin{lemma}\label{L:Bad part}
    \begin{equation}\label{E:Expression of bad part}
        A_2 = -\frac{d^2}{2} \HH_x \nabla^2 f_x (\bn_x , \bn_x) \pm \bigg( \delta \rho \langle \bn_y , \nabla f_x - \nabla f_y \rangle + \frac{\delta ^2 \rho}{\HH_x}\langle \bn_y ,\pa_n^x \rangle \nabla^2 f_x(\bn_x, \pa_n ^x) \  \bigg), 
    \end{equation}
    at any critical point of $\ZZ$. Here, $\nabla f_x$ is the gradient of $f$ in $\mb{R}^{n+1}$ at point $x$, $\nabla^2 f_x$ is the Hessian of $f$ in $\mb{R}^ {n+1}$ at point $x$, $\rho = \sqrt{1 + \frac{2\HH_x}{\delta^2} \ZZ }$ as before.
\end{lemma}

\begin{proof}
    By (\ref{E:Relation among normal vectors}), the first two terms of $A_2$ in (\ref{E:Expression of A_2}) are $\pm \langle T_x \bn_x - T_y \bn_y , \rho \delta \bn_y \rangle$. 
    The third term of $A_2$ in \ref{E:Expression of A_2}, by Lemma \ref{L:Evolution of T}, is:
    \begin{align*}
        \frac{d^2}{2} \big( - \HH_x \nabla^2 f_x (\bn_x , \bn_x) + \langle \nabla f_x , \nabla ^M \HH_x \rangle \big) \ .
    \end{align*}
    By (\ref{E:Gradient of H tilde}),(\ref{E:Relation among normal vectors}), and the fact that $\om$ is in the plane formed by $\{\pa_n ^x , \bn_x\}$,
    \begin{align*}
        \langle \nabla f_x , \nabla^M \HH_x \rangle  =& \  \pa_i ^x f \pa_i ^x \HH_x =  \frac{2 \pa_i ^x f}{d} \langle \om , \HH_x \pa_i ^x - \delta h_{ip}^x \pa_p ^x \rangle
        \\  =& \ \frac{2}{d} \big(\pa_n ^x f \HH_x - \delta h^x(\nabla^M f_x , \pa_n ^x)\big)\langle \om , \pa_n ^x\rangle
        \\  =& \ \pm \frac{2 \rho \delta}{d^2 \HH_x}\big(\pa_n ^x f \HH_x - \delta h^x(\nabla^M f_x , \pa_n ^x)\big)\langle \bn_y , \pa_n ^x\rangle \ .
    \end{align*}
    We also notice that $\bn_y = \langle \bn_y , \pa_n ^x \rangle \pa_n ^x + \langle \bn_y , \bn_x \rangle \bn_x$. Hence, the sum of first three terms of $A_2$ is 
    \begin{equation}\label{E:First part of bad term}
        -\frac{d^2}{2} \HH_x \nabla^2 f_x (\bn_x , \bn_x) \pm \delta \rho \bigg(\langle \bn_y , \nabla f _x - \nabla f _y \rangle - \frac{\delta h^x(\nabla^M f_x , \pa_n ^x)}{\HH_x} \langle \bn_y ,\pa_n ^x \rangle \bigg) \ .
    \end{equation}
    For the last term in (\ref{E:Expression of A_2}), by (\ref{E:Relation among normal vectors}), and the the fact that $\nabla_{\pa_n ^x} \bn_x$ is tangent to $M$,
    \begin{equation}\label{E:Last part of bad term}
        \begin{split}
        \delta d \langle \om , \nabla^M T_x \rangle =& \ \pm \frac{\delta^2 \rho}{\HH_x} \langle \bn_y , \nabla^M T_x \rangle  = \ \pm \frac{\delta^2 \rho}{\HH_x} \langle \bn_y , \pa_n ^x(T_x) \pa_n ^x \rangle 
        \\  =& \ \pm \frac{\delta^2 \rho}{\HH_x} \langle \bn_y , \pa_n ^x \rangle \bigg(\langle \nabla_{\pa_n ^x} \nabla f ,\bn_x \rangle + \langle \nabla f , \nabla_{\pa_n ^x} \bn_x \rangle\bigg)
        \\  =& \ \pm \frac{\delta^2 \rho}{\HH_x} \langle \bn_y , \pa_n ^x \rangle \bigg(\langle \nabla_{\pa_n ^x} \nabla f ,\bn_x \rangle + \langle \nabla^M f , \nabla_{\pa_n ^x} \bn_x \rangle\bigg)
        \\  =& \ \pm \frac{\delta^2 \rho}{\HH_x} \langle \bn_y , \pa_n ^x \rangle \bigg( \nabla^2 f_x (\pa_n ^x , \bn _x) + h^x (\nabla^M f , \pa_n ^x)\bigg) .
        \end{split}
    \end{equation}
    Sum equations (\ref{E:First part of bad term}) and (\ref{E:Last part of bad term}), we get the expression (\ref{E:Expression of bad part}).
\end{proof}

\begin{remark}
    Note that, by (\ref{E:Relation among normal vectors}), 
        $$ \pm \frac{\delta ^2 \rho}{\HH_x} \langle \bn_y , \pa_n ^x \rangle = \langle d \delta \om , \pa_n ^x \rangle \ , $$
    and with (\ref{E:inner product omega nu_Y}),
        $$\langle  d \delta \om , \bn_x \rangle = \big\langle d \delta \om , (\pm \rho \bn_y) - \frac{d \HH_x}{\delta} \om \big\rangle = -\frac{d^2 \HH_x}{2} + \ZZ \ .$$
    Hence, by(\ref{E:Expression of bad part}), we can alo get
        $$A_2 = -\ZZ \nabla^2 f_x (\bn_x , \bn_x) + \nabla^2 f_x (\bn_x , d \delta \om) + \langle \delta \bn_x + d \HH_x \om , \nabla f_x - \nabla f_y \rangle \ .$$
    This is an expression without term $\rho$,  but expression (\ref{E:Expression of bad part}) is easier to handle when we prove the Theorem \ref{T:Shrinker delta non collapsed}.
\end{remark}

\begin{proof}[Proof of Theorem \ref{T:Shrinker delta non collapsed}]
    If the initial data has positive rescaled mean curvature, then, by Lemma \ref{L:Preservation of rescaled mean curvature}, $\HH_x > 0$ at any point of space and time. For $f = -\frac{| x |^2}{4}$, we have $\nabla f = - \frac{x}{2}$ and $\nabla^2  f = - \frac{1}{2}\mr{I}$. Hence,
    \begin{equation}
        \nabla^2 f_x (\bn_x , \pa_n ^x) = 0 \ ,\nabla^2 f_x (\bn_x , \bn_x) = -\frac{1}{2}  \ ,
    \end{equation}
    and by equation (\ref{E:inner product omega nu_Y}),
    \begin{align*}
        \delta \rho \langle \bn_y , \nabla f_x - \nabla f_y \rangle = \frac{\delta \rho}{2} \langle \bn_y , F(y,t) - F(x,t) \rangle = \frac{d \delta \rho}{2} \langle \bn_y , \om \rangle = \pm \frac{d^2 \HH_x}{4}\bigg(1 + \frac{2 \ZZ}{d^2 \HH_x}\bigg) \ .
    \end{align*}
    By Lemma \ref{L:Bad part}, 
    \begin{equation}\label{Compare A_2}
        A_2 = \frac{d^2 \HH_x}{2} + \frac{\ZZ}{2} \geq \frac{\ZZ}{2} \ .
    \end{equation}
    Hence, by our previous calculation with Lemma \ref{L:Good part} and Lemma \ref{L:Bad part}, 
    \begin{equation}\label{Compare Z}
         \pa_t \ZZ - \mc{L}(\ZZ) \geq \bigg( \frac{1}{2} + |h^x| ^2 + \frac{4\HH (\HH - \delta h_{nn} ^x)}{\delta^2} \langle\om , \pa_n ^y \rangle ^2\bigg)\ZZ
    \end{equation}
    at any critical point of $\ZZ$.
    The maximum principle on $M \times M \setminus \{x=y\}$ implies that $\ZZ$ stays nonnegative if it is initially nonnegative ($\ZZ$ is zero on the diagonal $\{x=y\}$).
    
    If the initial data has negative rescaled mean curvature, we could also get the corresponding rescaled-$-\delta$-non-collapsed result by Lemma \ref{L:Equivalence of nonnegative Z tilde} with the same calculation above for some $\delta > 0$. More precisely, change the direction of inequalities in (\ref{Compare A_2}) and (\ref{Compare Z}). 
\end{proof}

\begin{remark}\label{R: Rescaling non collapsed}
    It is easy to check that the above calculation in the proof of Theorem \ref{T:Shrinker delta non collapsed} also holds for $f = - \frac{|x- \sigma x_0|^2}{4\sigma^2}, \ \forall \sigma \in (0,\infty), \ \forall x_0 \in \mb{R}^n$.
\end{remark}


\begin{thebibliography}{DGO}
\def\bi#1{\bibitem[#1]{#1}}
\bi{AL86} Abresch, U.; Langer, J. {\it The normalized curve shortening flow and homothetic solutions.} J. Differential Geom. 23 (1986), no. 2, 175–196. 
\bi{A12} Andrews, Ben. {\it Noncollapsing in mean-convex mean curvature flow.} Geom. Topol. 16 (2012), no. 3, 1413–1418. 
\bi{A92} Angenent, Sigurd B. {\it Shrinking doughnuts.} Nonlinear diffusion equations and their equilibrium states, 3 (Gregynog, 1989), 21–38, Progr. Nonlinear Differential Equations Appl., 7, Birkhäuser Boston, Boston, MA, 1992.
\bi{AIC95} Angenent, S.; Ilmanen, T.; Chopp, D. L. {\it A computed example of nonuniqueness of mean curvature flow in $R^3$.} Comm. Partial Differential Equations 20 (1995), no. 11-12, 1937–1958.
\bi{A10} Au, Thomas Kwok-Keung. {\it On the saddle point property of Abresch-Langer curves under the curve shortening flow.} Comm. Anal. Geom. 18 (2010), no. 1, 1–21. 
\bi{BW16} Bernstein, Jacob; Wang, Lu. {\it A sharp lower bound for the entropy of closed hypersurfaces up to dimension six.} Invent. Math. 206 (2016), no. 3, 601–627.
\bi{B16} Brendle, Simon. {\it Embedded self-similar shrinkers of genus $0$.} Ann. of Math. (2) 183 (2016), no. 2, 715–728. 
\bi{BH16} Brendle, Simon; Huisken, Gerhard. {\it Mean curvature flow with surgery of mean convex surfaces in $R^3$.} Invent. Math. 203 (2016), no. 2, 615–654.
\bi{BH18} Brendle, Simon; Huisken, Gerhard. {\it Mean curvature flow with surgery of mean convex surfaces in three-manifolds.} J. Eur. Math. Soc. (JEMS) 20 (2018), no. 9, 2239–2257.
\bi{CM12} Colding, Tobias H.; Minicozzi, William P., II. {\it Generic mean curvature flow I: generic singularities.} Ann. of Math. (2) 175 (2012), no. 2, 755–833. 
\bi{CIMW13} Colding, Tobias Holck; Ilmanen, Tom; Minicozzi, William P., II; White, Brian. {\it The round sphere minimizes entropy among closed self-shrinkers.} J. Differential Geom. 95 (2013), no. 1, 53–69.
\bi{CM15} Colding, Tobias Holck; Minicozzi, William P., II {\it Uniqueness of blowups and Łojasiewicz inequalities.} Ann. of Math. (2) 182 (2015), no. 1, 221–285.
\bi{CM18} Colding, Tobias H.; Minicozzi, William P., II. {\it Dynamics of closed singularities}. preprint, arXiv:1808.03219. To appear in Annales de l’Institut Fourier.
\bi{CM18-2} Colding, Tobias H.; Minicozzi, William P., II. {\it Wandering singularities.} preprint, arXiv:1809.03585. To appear in J. Differential Geom.
\bi{EW87} Epstein, C. L.; Weinstein, M. I.{\it A stable manifold theorem for the curve shortening equation.}
Comm. Pure Appl. Math. 40 (1987), no. 1, 119–139.
\bi{HK17} Robert, Haslhofer; Bruce, Kleiner. {\it Mean curvature flow of mean convex hypersurfaces.} Comm. Pure Appl. Math. 70 (2017), no. 3, 511–546.
\bi{HK17-2} Haslhofer, Robert; Kleiner, Bruce. {\it Mean curvature flow with surgery.} Duke Math. J. 166 (2017), no. 9, 1591–1626.
\bi{H02} Hatcher, Allen. {\it Algebraic topology.} Cambridge University Press, Cambridge, 2002. xii+544 pp. ISBN: 0-521-79160-X; 0-521-79540-0
\bi{HW19} Hershkovits, Or; White, Brian. {\it Sharp entropy bounds for self-shrinkers in mean curvature flow.} Geom. Topol. 23 (2019), no. 3, 1611–1619.
\bi{H84} Huisken, Gerhard. {\it Flow by mean curvature of convex surfaces into spheres.} J. Differential Geom. 20 (1984), no. 1, 237–266. 
\bi{H90} Huisken, Gerhard. {\it Asymptotic behavior for singularities of the mean curvature flow.} J. Differential Geom. 31 (1990), no. 1, 285–299. 
\bi{HS99} Huisken, Gerhard; Sinestrari, Carlo. {\it Mean curvature flow singularities for mean convex surfaces.} Calc. Var. Partial Differential Equations 8 (1999), no. 1, 1–14.
\bi{HS99-2} Huisken, Gerhard; Sinestrari, Carlo. {\it Convexity estimates for mean curvature flow and singularities of mean convex surfaces.} Acta Math. 183 (1999), no. 1, 45–70.
\bi{I95} Ilmanen, Tom. {\it Singularities of mean curvature flow of surfaces.} preprint. \url{http://people.math.ethz.ch/~ilmanen/papers/sing.ps}
\bi{KKM18} Kapouleas, Nikolaos; Kleene, Stephen James; Møller, Niels. {\it Martin Mean curvature self-shrinkers of high genus: non-compact examples.} J. Reine Angew. Math. 739 (2018), 1–39.
\bi{M11} Mantegazza, Carlo. {\it Lecture notes on mean curvature flow.} Progress in Mathematics, 290. Birkhäuser/Springer Basel AG, Basel, 2011. xii+166 pp. ISBN: 978-3-0348-0144-7 
\bi{N14} Nguyen, Xuan Hien. {\it Construction of complete embedded self-similar surfaces under mean curvature flow, Part III.} Duke Math. J. 163 (2014), no. 11, 2023–2056.
\bi{SW09} Sheng, Weimin; Wang, Xu-Jia. {\it Singularity profile in the mean curvature flow.} Methods Appl. Anal. 16 (2009), no. 2, 139–155.
\bi{S04} Sullivan, Dennis {\it Ren\'e Thom's work on geometric homology and bordism.} Bull. Amer. Math. Soc. (N.S.) 41 (2004), no. 3, 341–350.
\bi{S18} Sun, Ao. {\it Singularities of Mean
Curvature Flow of Surfaces with Additional Forces.} preprint, arXiv:1808.03937
\bi{W94} White, Brian. {\it Partial regularity of mean-convex hypersurfaces flowing by mean curvature.} Internat. Math. Res. Notices 1994, no. 4, 186 ff., approx. 8 pp.
\bi{W00} White, Brian. {\it The size of the singular set in mean curvature flow of mean-convex sets.} J. Amer. Math. Soc. 13 (2000), no. 3, 665–695. 
\bi{W03} White, Brian. {\it The nature of singularities in mean curvature flow of mean-convex sets.} J. Amer. Math. Soc. 16 (2003), no. 1, 123–138. 
\bi{W05} White, Brian. {\it A local regularity theorem for mean curvature flow.} Ann. of Math. (2) 161 (2005), no. 3, 1487–1519. 
\bi{Z16} Zhu, Jonathan J. {\it On the entropy of closed hypersurfaces and singular self-shrinkers.} preprint, arXiv:1607.07760. To appear in J. Differential Geom.
\end{thebibliography}
\end{document}